\documentclass[letterpaper, fleqn, 11pt]{amsart}
\usepackage[english]{babel}
\usepackage{graphicx,color}
\usepackage[all]{xy}
\usepackage{amsfonts,dsfont,mathrsfs}
\usepackage{amsmath, amsthm, amssymb}
\usepackage{enumerate, url}

\newcommand{\bbR}{\ensuremath{\mathbb{R}}}
\newcommand{\bbQ}{\ensuremath{\mathbb{Q}}}
\newcommand{\bbZ}{\ensuremath{\mathbb{Z}}}

\newcommand{\frg}{\ensuremath{\mathfrak{g}}}
\newcommand{\frh}{\ensuremath{\mathfrak{h}}}

\DeclareMathOperator{\diam}{diam}

\DeclareMathOperator{\isom}{Isom}

\DeclareMathOperator{\Out}{Out}
\DeclareMathOperator{\SL}{SL}

\numberwithin{equation}{section}

\newtheorem{thmnr}{Theorem}[section]

\newtheorem{propnr}[thmnr]{Proposition}

\newtheorem{lemnr}[thmnr]{Lemma}

\newtheorem{cornr}[thmnr]{Corollary}
\theoremstyle{definition}

\newtheorem{dfnnr}[thmnr]{Definition}

\newtheorem*{rmks}{Remarks}

\newtheorem{exnr}[thmnr]{Example}

\newtheorem*{claim}{Claim}

\newtheorem{probnr}[thmnr]{Problem}

\begin{document}

\title{Riemannian manifolds with local symmetry}

\author{Wouter van Limbeek}

\date{\today}

\begin{abstract}
We give a classification of many closed Riemannian manifolds $M$ whose universal cover $\widetilde{M}$ possesses a nontrivial amount of symmetry. More precisely, we consider closed Riemannian manifolds $M$ such that Isom$(\widetilde{M})$ has noncompact connected components. We prove that in many cases, such a manifold is as a fiber bundle over a locally homogeneous space. This is inspired by work of Eberlein (for nonpositively curved manifolds) and Farb-Weinberger (for aspherical manifolds), and generalizes work of Frankel (for a semisimple group action). As an application, we characterize simply-connected Riemannian manifolds with both compact and finite volume noncompact quotients.
\end{abstract}

\maketitle

\tableofcontents

\section{Introduction}
\label{sec:intr}

A basic question in Riemannian geometry is to understand those manifolds whose metric is highly symmetric. As a measure of the degree of symmetry we can use the size of the isometry group. On the one hand manifolds with large isometry group, such as compact symmetric spaces, play important roles in many areas of mathematics, but on the other hand there are many closed manifolds equipped with very special geometric structures with small isometry groups. For example, if $M$ is closed and hyperbolic, Isom$(M)$ is finite. However, the symmetry of the hyperbolic metric becomes apparent when one considers covers of $M$. Indeed, the universal cover of a hyperbolic manifold $M$ is $\widetilde{M}=\mathbb{H}^n$ so that $\isom(\widetilde{M})$ acts transitively on $\widetilde{M}$, revealing the symmetry of the hyperbolic metric.

In general, any isometry between covers of $M$ lifts to an isometry of $\widetilde{M}$, so that $\widetilde{M}$ is the most symmetric cover of $M$. Therefore a natural problem is to describe manifolds $M$ for which $G:=\isom(\widetilde{M})$ is large in some meaningful way. More precisely, it is well-known that $G$ is a Lie group with possibly infinitely many components, and containing the deck group $\pi_1(M)$ as a discrete subgroup. Further if $K$ is a compact connected group acting isometrically on $M$, then there is a cover $\overline{K}$ acting isometrically on $\widetilde{M}$. Since we cannot hope to usefully classify compact group actions on manifolds, we should at least assume that $G^0$ is noncompact if we aim for a classification. Therefore the most general classification problem that we can hope to answer is
\begin{probnr} Describe the closed Riemannian manifolds $M$ such that $G^0$ is noncompact. \label{prob:symm}\end{probnr}
A sufficiently explicit solution of this problem has strong implications in many areas of mathematics. For example, as an application of our results we characterize simply-connected Riemannian manifolds $X$ that admit both a compact and a finite volume noncompact quotient (see Theorem \ref{cor:quots} below). This generalizes a result for contractible $X$ of Farb-Weinberger \cite{FW}.


We review progress on Problem \ref{prob:symm}. An answer has been obtained by Eberlein if $M$ is nonpositively curved \cite{eblatt, ebisom}, and more generally by Farb-Weinberger \cite{FW} if $M$ is aspherical. Melnick has considered Problem \ref{prob:symm} for real-analytic closed aspherical Lorentzian manifolds $M$ if Isom$(\widetilde{M})^0$ is semisimple \cite{melsslorentz}.

Roughly, the theorem of Farb-Weinberger states that if $M$ is as in Problem \ref{prob:symm}, and in addition aspherical, then a finite cover of $M$ is a fiber bundle such that the fibers are locally homogeneous spaces. For a precise statement some orbifold phenomena need to be accounted for, see Theorem \ref{thm:fw} below. Previously this was known for nonpositively curved manifolds by work of Eberlein, who used tools coming from differentialy geometry in nonpositive curvature. In this case the orbifold phenomena are not present, and a genuine fiber bundle is obtained that is actually a Riemannian product.

Neither of these methods generalize to the nonaspherical case, and in fact, the classification result obtained by Farb-Weinberger is false without the asphericity assumption (see Remark \ref{rmks:thm}.3 and Example \ref{ex:fwfails} below). The first progress for not necessarily aspherical manifolds is due to Frankel \cite{frharm}, who proved that under the assumption that $G^0$ is semisimple without compact factors and with finite center, $M$ has a finite cover that fibers over a locally symmetric space of noncompact type.

So all the progress on Problem \ref{prob:symm} either relies on topological restrictions on $M$ or algebraic restrictions on $G^0$.
The goal of this paper is to solve Problem \ref{prob:symm} when $G$ has finitely many components. In addition we reduce the general case to $G^0$ being solvable. In particular, our results make no assumptions on the topology of $M$, other than that is closed, and we do not require $G^0$ to be semisimple. To state our results precisely, we recall that $G$ has a Levi decomposition.
\subsection*{Levi Decomposition:} (see \cite{raghlie}) Let $H$ be a connected Lie group. Then there exist
		\begin{itemize}
			\item a unique maximal connected closed normal solvable subgroup $R$ of $H$ (the \emph{solvable radical}) and
			\item a  closed, connected semisimple subgroup $S$ of $H$ 
		\end{itemize}
	such that $H=RS$.  Further $S$ is unique up to conjugation. Any such subgroup $S$ is called a \emph{Levi subgroup} of $H$. \medskip
	
We also introduce the following terminology.
	\begin{dfnnr}
	\label{dfn:fibhom}
	Let $M$ be a closed Riemannian manifold. We say $M$ \emph{fibers locally equivariantly over a locally homogeneous space} $B=X\slash\Lambda$, where $X$ is homogeneous and simply-connected, if there exist
	\begin{enumerate}[(i)]
		\item a smooth fiber bundle $p:M\rightarrow B$,
		\item a closed subgroup $H\subseteq \textrm{Isom}(\widetilde{M})$ containing $\pi_1(M)$, and
		\item a smooth group homomorphism $\varphi:H\rightarrow \textrm{Isom}(X)$ with image acting transitively on $X$
	\end{enumerate}
	such that there is a lift $\widetilde{p}:\widetilde{M}\rightarrow X$ of $p$ that is $\varphi$-equivariant. When there is no danger of confusion, we will often omit the locally homogeneous space $B$, and we will just say that $M$ fibers locally equivariantly.
	We say $M$ \emph{virtually fibers locally equivariantly} over a locally homogeneous space $B$ if a finite cover of $M$ fibers locally equivariantly over $B$.
	\end{dfnnr}
The power of a locally equivariant fibering comes from the equivariance with respect to a morphism $H\rightarrow \varphi(H)$ where $\varphi(H)$ acts transitively on $X$. For example, this implies that the fibers of $\widetilde{p}$ are isometric with respect to the induced metric. Our main result is the following.
	\begin{thmnr} Let $M$ be a closed Riemannian manifold. Set $G:=\textrm{Isom}(\widetilde{M})$ and let $G^0=RS$ be the Levi decomposition of $G^0$. Then either
		\begin{enumerate}[(i)]
			\item $G^0$ is compact,
			\item $S$ is compact and $G$ has infinitely many components, or
			\item $M$ virtually fibers locally equivariantly over a locally homogeneous space.
		\end{enumerate}
	\label{thm:cases}
	\end{thmnr}
In particular the theorem applies when $G$ has finitely many components. We now give a much more detailed description in this case.
	\begin{cornr}
	\label{cor:ind}
	 Let $M$ be a closed Riemannian manifold such that $G=\isom(\widetilde{M})$ is noncompact and has finitely many components. Then
		\begin{enumerate}[(i)]
			\item $M$ virtually fibers locally equivariantly over a locally symmetric space $X$ of noncompact type (possibly a point) with isometric fibers $F_1$. Further, $X$ is a point if and only if $S$ is compact,
			\item $F_1$ virtually fibers locally equivariantly over a compact torus $T^n$ (possibly $n=0$) with isometric fibers $F_2$,
			\item $F_2$ virtually fibers locally equivariantly over a nilmanifold $N$ (possibly a point) with isometric fibers $F_3$. Further $\dim N$ and $n$ are both zero if and only if $R$ is compact, and
			\item if $Z(S)$ is finite, then Isom$(\widetilde{F_3})$ is compact.
		\end{enumerate}
	\end{cornr}
This result should be viewed as a decomposition of $M$ into homogeneous directions (the base spaces in statements (i)-(iii)) and into a residual part (the fibers $F_3$). The conditions for (non)triviality of certain base spaces guarantees that Lie theoretic properties of $G$ are reflected in the geometry of $M$.
	\begin{rmks}\mbox{}
		\begin{enumerate}
		\item Note that there are examples in which $G$ has infinitely many components. For example,consider the lift of a very bumpy metric on an $n$-torus to its universal cover $\mathbb{R}^n$. It is easy to see the deck transformations will be finite index in the isometry group of $\mathbb{R}^n$.
		\item One may hope to generalize Theorem \ref{thm:cases} to also give a description of $M$ that belong to cases (i) or (ii). However, the problem in this generality is hopeless: any action by a connected compact Lie group on $M$ admits an invariant metric and lifts to an action of some group on the universal cover $\widetilde{M}$. Therefore a complete solution of (i) or (ii) would classify compact group actions on manifolds.
		\item Even though the flavor of Theorem \ref{thm:cases} is similar to that of the result of Farb-Weinberger, in the nonaspherical case new phenomena appear. As a consequence, the result is different. The fiber bundle constructed in \cite{FW} has locally homogeneous fibers. Indeed the fibers are projections to $M$ of $\isom(\widetilde{M})^0$-orbits in $\widetilde{M}$. In contrast, the fiber bundle in Theorem \ref{thm:cases} has locally homogeneous base. Indeed, in Section \ref{sec:ex} we give an example of a nonaspherical manifold such that $G$ is not discrete (in fact has finitely many components), but it is not a Riemannian orbibundle such that the fibers are projections of orbits of any closed subgroup $H\subseteq G$ containing $\Gamma$.
		\end{enumerate}
	\label{rmks:thm}
	\end{rmks}
As an application of Theorem \ref{thm:cases}, we characterize Riemannian manifolds that admit both a compact and a finite volume noncompact quotient.
\begin{cornr} Let $X$ be a simply-connected Riemannian manifold that admits both a compact quotient $M$ and a finite volume noncompact quotient $N$. Then $M$ virtually locally equivariantly fibers over a locally symmetric space of noncompact type.
\label{cor:quots}
\end{cornr}
So the only source of Riemannian manifolds with both compact and finite volume noncompact quotients are symmetric spaces. This result generalizes a result of Farb-Weinberger for contractible manifolds $X$.

Since Problem \ref{prob:symm} is situated in the Riemannian category, it is natural to ask whether the fiber bundle obtained in Theorem \ref{thm:expl} is a Riemannian submersion. However, this need not be true (see Example \ref{ex:norsub}). Further, while for any virtual locally equivariant fibering $p$, the fibers of $\widetilde{p}:\widetilde{M}\rightarrow X$ are isometric, the fibers of $p$ need not be isometric (see Example \ref{ex:isomfiber}). Still, we have the following regularity result for $p:M\rightarrow X$ in the Riemannian category.
 	\begin{thmnr}
	\label{thm:riemsub}
	Let $M$ be a closed Riemannian manifold that virtually fibers locally equivariantly over a locally homogeneous space. In the notation of definition \ref{dfn:fibhom}, $M$ admits a Riemannian metric such that $H$ acts isometrically on $\widetilde{M}$ and $p$ is a Riemannian submersion with totally geodesic fibers.
	\end{thmnr}
	
The claims of Theorem \ref{thm:cases} follow from a more detailed result that explicitly constructs the locally homogeneous space, given the assumptions of case (iii).
	\begin{thmnr} 
	\label{thm:expl}
	Let $M$ be a closed Riemannian manifold, and let $G:=\textrm{Isom}(\widetilde{M})$. Assume that $G^0$ is noncompact. Let $G^0=RS$ be the Levi decomposition of $G^0$. Then
		\begin{enumerate}[(i)]
			\item if $S$ is noncompact then $M$ virtually fibers locally equivariantly over a nontrivial locally symmetric space of noncompact type.
			\item if $S$ is compact and $G$ has finitely many components, then $M$ virtually fibers locally equivariantly over a nontrivial nilmanifold.
		\end{enumerate}
	\end{thmnr}
\subsubsection*{Outline of proof} In both cases, we produce a cocompact lattice $\Lambda$ in a Lie group $H$, a finite index subgroup $\Gamma'\subseteq \pi_1(M)$ and a morphism $\varphi:\Gamma'\rightarrow\Lambda$. In both cases $\varphi$ induces a map $f:M'\rightarrow N$ for a finite cover $M'$ of $M$, which is determined up to homotopy. Here $N$ is a locally symmetric space in Case (i) and a nilmanifold in Case (ii). Our goal is then to show $f$ can be chosen to be a locally equivariant fiber bundle. 

In case (i) we then use the theory of harmonic maps and a barycenter construction that generalizes both Proposition 3.1 of \cite{FW} and the work of Frankel \cite{frharm}. This method does not apply in case (ii), but instead we find $\Gamma_k\supseteq \Gamma'$ and repeat the construction above to find maps
	$$f_k: M_k:=\widetilde{M}\slash \Gamma_k\rightarrow N_k.$$
Using the Arzel\`a-Ascoli theorem, we can show lifts of $f_k$ to $\widetilde{M}$ converge uniformly to a map
	$$f:\widetilde{M}\rightarrow \widetilde{N}.$$
In general $f$ is not necessarily smooth, but just continuous. To remedy this, we approximate $f$ by an equivariant smooth map $p$ and show this approximation is a fiber bundle. This will prove Theorem \ref{thm:expl}.

The paper is organized as follows: In Section \ref{sec:prelim} we recall some preliminaries needed for the rest of the paper. In Section \ref{sec:pfexist} we prove the existence of the virtual locally equivariant fibering as claimed by Theorem \ref{thm:expl}, and the existence of the metric such that this fibering is a Riemannian submersion (Theorem \ref{thm:riemsub}). In Section \ref{sec:decomp} we prove the decomposition of Corollary \ref{cor:ind} and in Section \ref{sec:quots} we prove Corollary \ref{cor:quots} characterizing manifolds with uniform and nonuniform quotients. In Section \ref{sec:ex} we give examples that show that in the nonaspherical situation one cannot expect a result as proved in \cite{FW}, and we show that the fiber bundle obtained in Theorem \ref{thm:expl} need not be Riemannian, or have isometric fibers. Finally in Section \ref{sec:sea} we prove the generalization of the Mostow-Palais equivariant embedding theorem needed in the proof of Theorem \ref{thm:expl}.

\subsection*{Acknowledgments:}
I am pleased to thank Daniel Studenmund, Bena Tshishiku, and Shmuel Weinberger for many helpful conversations and suggestions. Many thanks to Karin Melnick, Daniel Stundenmund and Bena Tshishiku for reading an earlier version of this paper and helpful comments. I am very grateful to my thesis advisor Benson Farb for his suggestion that a version of this project should be true, his invaluable advice during the project, and continuing unbounded enthusiasm in all matters mathematical. I would like to thank the University of Chicago for support during the work on this project.

\section{Preliminaries}
\label{sec:prelim}

\subsection{Lie theory}

%

In this section we review some Lie theoretic preliminaries needed for the proof of our results. Most of the results in this section and their proofs can be found in \cite{raghlie} or \cite{ovlie}. In the rest of this section, $G$ will be a connected Lie group with Levi decomposition $G=RS$. 

\subsubsection*{Lattice heredity.} Let $\Gamma$ be a lattice in $G$. Which morphisms $f:G\rightarrow H$ map $\Gamma$ to a lattice in $H$? When is $f(\Gamma)$ is discrete in $H$? Here we will restrict our attention to maps $f$ that arise as projections $G\rightarrow G\slash K$. To answer these questions, we introduce the following terminology.
	\begin{dfnnr}[\cite{ovlie}] Let $H\subseteq G$ be a closed subgroup. We say $H$ is \emph{(uniformly) $\Gamma$-hereditary} if $\Gamma\cap H$ is a (cocompact) lattice in $H$. We say $H$ is \emph{$\Gamma$-closed} if the set $H\Gamma$ is closed in $G$.
	\end{dfnnr}
Then we have the following criteria for $\Gamma$ to project to a lattice.
	\begin{propnr}[{\cite[1.4.3]{ovlie}} and {\cite[1.4.7]{ovlie}}] Let $H$ be a closed normal subgroup of $G$, and let $\Gamma$ be a (cocompact) lattice in $G$. Then the following are equivalent:
		\begin{enumerate}[(i)]
			\item $\Gamma$ projects to a (cocompact) lattice in $G\slash H$. 
			\item $H$ is (uniformly) $\Gamma$-hereditary.
			\item $H$ is $\Gamma$-closed.
		\end{enumerate}
	\label{prop:latther}
	\end{propnr}
	
\subsubsection*{$\Gamma$-(non)heredity of radicals:} In general $R$ need not be $\Gamma$-hereditary.
	\begin{exnr} Set $G=\mathbb{R}\times SO(3)$ and let $A$ be an infinite order element of $SO(3)$. Consider the representation
		$$\rho:\mathbb{Z}\rightarrow G$$
	defined by $\rho(1)=(1,A)$. Then $\rho(\mathbb{Z})$ is a lattice in $G$ but $\rho(\mathbb{Z})\cap \mathbb{R}=0$ is not a lattice in $\mathbb{R}$.
	\label{ex:rhered}
	\end{exnr}
However, there is the following result due to Auslander.
	\begin{thmnr}[{Auslander, \cite[4.1.7(i)]{ovlie}}] Write $S=LH$ where $L$ is the maximal compact factor of $S$ and $H$ is the maximal noncompact factor of $S$. Let $\Gamma$ be a lattice in $G$. Then $\Gamma\cap RL$ is a lattice in $RL$.
	\label{thm:auslrl}
	\end{thmnr}
Because $RL$ is not solvable, one cannot expect $\Gamma\cap RL$ to be solvable. However, there is a solvable group in sight.
	\begin{thmnr}[{Auslander, \cite[8.24]{raghlie}}] Consider the closed subgroup $\overline{\Gamma R}$ of $G$. Its connected component of the identity $\overline{\Gamma R}^0$ is solvable.
	\label{thm:auslsolcl}
	\end{thmnr}
In Example \ref{ex:rhered} of a group $G$ and a lattice $\Gamma\subseteq G$ such that $R$ is not $\Gamma$-hereditary, any Levi subgroup of $G$ centralizes the radical $R$. This is in fact the only obstruction.
	\begin{thmnr}[{Mostow, \cite[1.3]{wunote}}] Let $G$ be a connected Lie group with Levi decomposition $G=RS$. Write $S=LH$ where $L$ is the maximal compact factor of $S$ and $H$ is the maximal noncompact factor of $S$, and assume that the centralizer of $R$ in $L$ is finite. Then for $\Gamma\subseteq G$ a lattice, $R$ is $\Gamma$-hereditary.
	\label{thm:mostr}
	\end{thmnr}
Besides $R$, there is another convenient subgroup of $G$ that we can work with, namely the \emph{nilradical} $N$ of $R$. The nilradical of $R$ is the unique maximal, connected, normal, nilpotent subgroup of $R$. Theorem \ref{thm:mostr} has an analogue for $N$.
	\begin{thmnr}[{Mostow, \cite[1.3]{wunote}}] Let $G$ be a connected Lie group with Levi decomposition $G=RS$ and let $N$ be the nilradical of $R$. Write $S=LH$ where $L$ is the maximal compact factor of $S$ and $H$ is the maximal noncompact factor of $S$, and assume that the centralizer of $R$ in $L$ is finite. Then for $\Gamma\subseteq G$ a lattice, $N$ is $\Gamma$-hereditary.
	\label{thm:mostn}
	\end{thmnr}
In particular the theorem applies if $G=R$ is solvable.

\subsubsection*{Lattices in nilpotent Lie groups.} 

When studying connected subgroups of Lie groups, the exponential map is an essential tool to transfer to the linear setting of Lie algebras. In nilpotent groups this strategy also applies when studying lattices.
	\begin{thmnr} Let $G$ be a simply-connected nilpotent group. Then the exponential map $\exp:\frg\rightarrow G$ is a diffeomorphism.
	\label{thm:expdiff}
	\end{thmnr}
In particular simply-connected nilpotent groups are contractible. The theorem suggests a strong correspondence between lattices in $G$ and $\frg$, and this is in fact true.
	\begin{thmnr}[{\cite[2.12]{raghlie}}] Let $G$ be a simply-connected nilpotent group. Then

		\begin{enumerate}[(i)]
			\item For any lattice $\Gamma\subseteq G$, the $\mathbb{Z}$-span of $\exp^{-1}\Gamma$ is a lattice (i.e.\ additive subgroup of maximal rank) in $\frg$. Further, the structural constants of $\frg$ with respect to a basis for $\exp^{-1}\Gamma$ are rational.
			\item Let $\frg_0$ be the $\mathbb{Q}$-span of a basis for $\frg$ with respect to which the structural constants are rational. Then for any lattice $\mathcal{L}$ in $\frg_0$, $\exp\mathcal{L}$ is a lattice in $G$.
		\end{enumerate}
	\label{thm:nilplatt}
	\end{thmnr}
In general it is not true that $\exp^{-1}\Gamma$ itself is a lattice in $\frg$.
	\begin{dfnnr} Let $G$ be a simply-connected nilpotent Lie group. We say a lattice $\Gamma\subseteq G$ is an \emph{exponentiated lattice} if $\exp^{-1}\Gamma$ is an additive subgroup of $\frg$.\end{dfnnr}
The correspondence provided by Theorem \ref{thm:nilplatt} is one-to-one up to finite index subgroups.
	\begin{thmnr}[{\cite[2.2.13]{ovlie}}] Let $\Gamma$ be a lattice in a simply-connected Lie group $G$. Then there are exponentiated lattices $\Gamma_1, \Gamma_2$ in $G$ such that $\Gamma_1\subseteq \Gamma\subseteq \Gamma_2$. Further both inclusions are as finite index subgroups.
	\label{thm:lattsubgp}
	\end{thmnr}
Of course Theorem \ref{thm:expdiff} fails if $G$ is not simply-connected. However, this failure is very controlled.
	\begin{thmnr} Let $G$ be a connected nilpotent group. Then $G$ has a unique maximal compact torus $T$. Further $T$ is central and $G\slash T$ is simply-connected.
	\label{thm:nilptorus}
	\end{thmnr}

\subsubsection*{Lattices in semisimple groups.} Let $G$ be a connected semisimple Lie group. Its center $Z(G)$ is a closed abelian normal subgroup, hence it must be discrete. Using this fact, it is easy to see that $G\slash Z(G)$ has trivial center. Another classical fact we will need is that $\Out(G)$ is finite.

Finally, we note that the theme of a lattice resembling the ambient group is especially powerful in the context of semisimple groups (e.g. the Mostow rigidity theorem). Here we will just note that the center of a lattice resembles the center of the ambient group, which follows from a version of the Borel density theorem:
	\begin{thmnr}[Borel, {\cite[5.17, 5.18]{raghlie}}] Let $G$ be a semisimple group without compact factors and $\Gamma\subseteq G$ a lattice. Then $Z(\Gamma)=\Gamma\cap Z(G)$, and $Z(\Gamma)$ is a lattice in $Z(G)$.
	\label{thm:boreldensity}
	\end{thmnr}

\subsection{Harmonic maps}

A second ingredient in the proofs below is the theory of harmonic maps. For a smooth map $f:M\rightarrow N$ between Riemannian manifolds $M,N$, we define the energy
	$$E(f):=\int_M ||Df||^2.$$
We say $f$ is \emph{harmonic} if it locally minimizes the energy. There is a very rich theory of these maps and their rigidity, but we only mention the following existence and uniqueness results:
	\begin{thmnr}[{\cite{synpc}}] Let $M,N$ be closed Riemannian manifolds.
		\begin{enumerate}[(i)]
			\item \emph{(Eells-Sampson)} If $N$ is nonpositively curved, then there exists a harmonic map $M\rightarrow N$ in each homotopy class.
			\item \emph{(Hartman, Schoen-Yau)} Suppose $f:M\rightarrow N$ is harmonic, surjective on $\pi_1$ and that $Z(\pi_1(N))=1$. Then $f$ is the unique harmonic map in its homotopy class.
		\end{enumerate}
	\label{thm:harm}
	\end{thmnr}

\section{Existence of locally equivariant fiberings}
\label{sec:pfexist}

We introduce some notation that will be used below. Let $M$ be a closed Riemannian manifold and set $\Gamma:=\pi_1(M)$. Let $G:=\textrm{Isom}(\widetilde{M})$. Let $R$ be the solvable radical of $G^0$ and $S':=G^0\slash R$. Set $\Gamma_0:=\Gamma\cap G^0$. 

%

We give an outline of the proof of Theorem \ref{thm:expl}.(i). Assume that $S'$ is noncompact. The first part of the proof is entirely Lie theoretic. We will construct a cocompact lattice $\Lambda_1$ in a semisimple Lie group $S_1$ with no compact factors and trivial center, a finite index subgroup $\Gamma'\subseteq\Gamma$, and a morphism $\varphi:\Gamma'\rightarrow\Lambda_1$ (Lemmas \ref{lem:qi} and \ref{lem:split}). 

The second part of the proof is geometric, and relates the morphism $\varphi:\Gamma'\rightarrow\Lambda_1$ to the geometry of $M$. By asphericity of the locally symmetric space $X$ associated to $\Lambda_1$, there is a smooth map $f:M'\rightarrow X$ (where $M'$ is the cover corresponding to $\Gamma'$) inducing $\varphi$ on fundamental groups. 

The map $f$ is only determined up to homotopy. We use the theory of harmonic maps between Riemannian manifolds to select a unique harmonic representative of the homotopy class of $f$. The harmonicity will imply that this map is a fiber bundle.

\subsection*{Proof of Theorem \ref{thm:expl}.(i)} Assume $S'$ is noncompact.

\subsection*{Step 1 (Lie theory):} First we prove:
	\begin{lemnr}
	\label{lem:qi}
	$\Gamma_0\subseteq G^0$ is a cocompact lattice.
	\end{lemnr}

\begin{proof}[Proof of Lemma \ref{lem:qi}]
Because $M$ is compact, by the Milnor-Schwarz lemma we have that for any $x\in \widetilde{M}$ the map
	\begin{align*}
			\Gamma&\rightarrow \widetilde{M}\\
			\gamma&\mapsto \gamma\cdot x
	\end{align*}
is a quasi-isometry. Since $G$ also acts on $\widetilde{M}$ properly, isometrically, and contains $\Gamma$, the map
	\begin{align*}
G&\rightarrow \widetilde{M}\\
g&\mapsto g\cdot x
	\end{align*}
is also a quasi-isometry. It follows that the inclusion $\Gamma\hookrightarrow G$ is a quasi-isometry, so $\Gamma$ is a cocompact lattice in $G$. Since $G^0$ is a connected component of $G$, $G^0\slash\Gamma_0$ is closed in $G\slash\Gamma$. Because $G\slash\Gamma$ is compact, it follows $G^0\slash\Gamma_0$ is compact, as desired.
\end{proof}

We can use $\Gamma_0$ to obtain a lattice in a semisimple group. Let $S$ be a Levi subgroup of $G^0$, so that $S'=S\slash(R\cap S)$. Let $L$ be the unique maximal connected closed normal compact subgroup of $S$. By Theorem \ref{thm:auslrl}, $\Gamma\cap RL$ is a lattice in $RL$. It is necessarily cocompact by Proposition \ref{prop:latther}, and again by Proposition \ref{prop:latther}, $\Gamma$ projects to a cocompact lattice in $G^0\slash RL$. 

$G^0\slash RL$ is a semisimple group, but it may have infinite center. However, below we need to have a semisimple group without center, so consider $S_1:=(G^0\slash RL)\slash Z$, where $Z:=Z(G^0\slash RL)$. Then $S_1$ is a semisimple group without center, and since $\Gamma\cap Z$ is a (necessarily cocompact) lattice in $Z$ by Theorem \ref{thm:boreldensity}, the image of $\Gamma_0$ in $S_1$ is a cocompact lattice.

We will extend the projection map $G^0\rightarrow S_1$ to a map $G'\rightarrow S_1$ for a suitable finite index subgroup $G'\subseteq G$. Note that this is nontrivial because $G$ may have infinitely many components. In order to obtain the extension, set $G_1:=\langle G^0,\Gamma\rangle$. Then $G_1$ is a finite index subgroup of $G$. Consider the short exact sequence
	\begin{equation}
		1\rightarrow S_1 \rightarrow (G_1\slash RL)\slash Z \rightarrow \Gamma\slash\Gamma_0\rightarrow 1.
		\label{eq:ses}
	\end{equation}
This sequence is almost a direct product.
	\begin{lemnr}
		\label{lem:split}
		There is a finite index subgroup $\Gamma'\subseteq\Gamma$ containing $\Gamma_0$ such that the induced sequence 
	\begin{equation}
		\label{eq:ses2}
		1\rightarrow S_1\rightarrow (G_1'\rangle\slash RL)\slash Z\rightarrow \Gamma'\slash\Gamma_0\rightarrow 1
	\end{equation}
splits as a direct product, where $G_1':=\langle G^0,\Gamma'\rangle$.
\end{lemnr}
\begin{proof}
The sequence \ref{eq:ses} is determined by
	\begin{enumerate}[(i)]
		\item A representation $\rho: \Gamma\slash\Gamma_0 \rightarrow \textrm{Out}(S_1)$, and
		\item A cohomology class in $H^2(\Gamma\slash\Gamma_0, Z(S_1)_\rho)$, where the module structure on $Z(S_1)$ is induced by $\rho$.
	\end{enumerate}
Note that since $S_1$ has no center, the cohomology class in (ii) vanishes. Since $\Out(S_1)$ is finite, there is a finite index subgroup $\Gamma'$ of $\Gamma$ containing $\Gamma_0$ such that $\rho$ is trivial on $\Gamma'\slash\Gamma_0$. Hence for this choice of $\Gamma'$, the sequence \ref{eq:ses2} splits as a direct product.
\end{proof}

Since the statement of Theorem \ref{thm:expl}.(i) is an assertion about a finite cover of $M$, we can replace $\Gamma$ by $\Gamma'$. Then $G':=\langle G^0,\Gamma\rangle$ is a finite index subgroup of $G$ containing $G^0$, and by Lemma \ref{lem:split} we have
	$$(G'\slash RL)\slash Z \cong S_1\times (\Gamma\slash\Gamma_0).$$
Hence we can extend the projection $G^0\rightarrow S_1$ to $\varphi: G'\rightarrow S_1$, and $\varphi$ maps $\Gamma$ surjectively onto a cocompact lattice $\Lambda_1$ of $S_1$. 

\subsection*{Step 2 (Geometric):} Let $K\subseteq S_1$ be a maximal compact subgroup and consider the locally symmetric space $X:=\Lambda_1\backslash S_1\slash K$. Because $X$ is aspherical, the homomorphism $\Gamma\rightarrow \Lambda_1$ induces a map $f:M\rightarrow X$ determined up to homotopy. To select a representative of this homotopy class that will be a fiber bundle, we use the theory of harmonic maps between Riemannian manifolds. 

In our situation, $X$ is nonpositively curved, so that by the theorem of Eells-Sampson (Theorem \ref{thm:harm}.(i)), there is a harmonic representative in the homotopy class of $f$. From now on, we will denote this harmonic representative by $f$.

Note that $\Gamma\rightarrow \Lambda_1$ is surjective and by the Borel density theorem,
	$$Z(\Lambda_1)=\Lambda_1\cap Z(S_1)=1.$$
So by Theorem \ref{thm:harm}.(ii) of Hartman and Schoen-Yau, $f$ is the unique harmonic map in its homotopy class. We will show that this rigidity in the choice of $f$ implies that it is a fiber bundle. First lift $f$ to $\widetilde{f}:\widetilde{M}\rightarrow S_1\slash K$. Then we have:
\begin{lemnr}
	\label{lem:equivar}
	$\widetilde{f}$ is $\varphi$-equivariant. 
\end{lemnr}
\begin{proof}
This will follow from an `averaging' construction due to Frankel \cite{frharm}. Frankel carried out this construction only for $G=S_1$ and $\varphi=\textrm{id}_G$, but his proof applies verbatim here, so we will merely recall the construction.

The intuition for the construction is to compare the points $\varphi(g^{-1})\widetilde{f}(g x)$ and $f(x)$ for $x\in\widetilde{M}$ and $g\in G'$. Note that $\widetilde{f}$ is $\varphi$-equivariant precisely when these points always coincide, i.e.\ for every $g\in G'$ and $x\in\widetilde{M}$, we have
	\begin{equation}
		\varphi(g)^{-1}\widetilde{f}(gx)=\widetilde{f}(x).
		\label{eq:equivar}
	\end{equation}
To measure how far $\widetilde{f}$ is from being equivariant, for $x\in\widetilde{M}$, set
	$$O_x:=\left\{\varphi(g)^{-1}\widetilde{f}(g x) \mid g\in G'\right\}\subseteq \widetilde{X}.$$
Because $f$ induces the map $\varphi|_{\Gamma}:\Gamma\rightarrow\Lambda_1$ on $\pi_1$ and $\widetilde{f}$ is a lift of $f$, we know that $\widetilde{f}$ is $\varphi|_{\Gamma}$-equivariant, so Equation $\ref{eq:equivar}$ holds for $g\in\Gamma$ and any $x\in\widetilde{M}$.  Because $\Gamma$ is a cocompact lattice in $G'$, it follows $O_x$ is compact. Frankel shows \cite[Thm 3.5]{frharm} that compactness of $O_x$ together with the the nonpositive curvature of $X$ implies that there is a well-defined unique barycenter $b(x)$ of the set $O_x$. The map $x\mapsto b(x)$ is then the `average' of $f$. More precisely, define for $x\in\widetilde{M}$ the map
	$$\widetilde{f}^x:\Gamma\backslash G'\rightarrow \widetilde{X},\qquad [g]\mapsto \varphi(g)^{-1}\widetilde{f}(gx).$$
This is well-defined because $\widetilde{f}$ is a lift of $f$, hence $\varphi|_\Gamma$-equivariant. Because $G'$ has a lattice, it is unimodular, so Haar measure $\mu$ on $G'$ descends to $\Gamma\backslash G'$. This gives a compactly supported measure $\nu_x:=\widetilde{f}^x_\ast \mu$ on $\widetilde{X}$. For $y\in \widetilde{X}$, define 
	$$D(y;x):=\int_{\widetilde{X}} d(y,z)^2 d\nu_x(z).$$
Because $\widetilde{X}$ is nonpositively curved and contractible, the function $z\mapsto d(y,z)^2$ is strictly convex. Together with the fact that $\nu_x$ is compactly supported, this implies that there is a unique $y\in \widetilde{X}$ with $\nabla_y D(y;x)=0$. The \emph{barycenter} of $\widetilde{f}^x$ is defined to be this point $y$, and is denoted $b(x)$.

It is not hard to see that the map $b$ is $\varphi$-equivariant. In particular, $b$ descends to a map
	$$\overline{b}:M\rightarrow X$$
and $\overline{b}$ induces the map $\varphi|_\Gamma$ on $\pi_1$. Because $X$ is aspherical, and both $\overline{b}$ and $f$ induce the map $\varphi|_\Gamma$ on $\pi_1$, it follows that $f$ and $\overline{b}$ are homotopic, but Frankel also proves \cite[Thm 3.3]{frharm} that the averaging construction decreases the energy. But because $f$ is the unique harmonic map in its homotopy class, it uniquely minimizes the energy functional in its homotopy class. Therefore $\overline{b}=f$. Hence $\widetilde{f}$ is $\varphi$-equivariant, as claimed.\end{proof}
	\begin{lemnr}
	\label{lem:fiber}
	$f$ is a fiber bundle.
	\end{lemnr}
\begin{proof}

First we show $f$ is a fibration. Equivalently, $f$ admits a path lifting function, i.e.\ it is possible to lift paths from $X$ to $M$ continuously. Since this is a local property, it suffices to prove $\widetilde{f}$ is a fibration. Since the composition
	$$G^0\rightarrow S_1\rightarrow S_1\slash K$$
is a smooth fiber bundle, there is a path lifting function from $S_1\slash K$ to $G^0$ that lifts smooth paths to smooth paths. This naturally induces a path lifting for $\widetilde{f}$. If $c$ is a curve in $S_1\slash K$ starting at $x_0$ and lifts to a curve $\widetilde{c}(t)$ in $G^0$, then for $\widetilde{x}_0$ a point in $\widetilde{M}$ over $x_0$, the path
	$$t\mapsto \widetilde{c}(t)\widetilde{c}(0)^{-1} \widetilde{x}_0$$
lifts $c$ to $\widetilde{M}$. It is clear that the constructed path depends continuously on $c$ and $\widetilde{x}_0$, so that we have a path-lifting function. It follows that $f$ is a fibration. Further we note that in this construction, smooth paths lift to smooth paths.

Finally, by a theorem of Ehresmann, a proper submersion is a fiber bundle. Since $M$ is compact, $f$ is clearly proper. It remains to show $f$ is a submersion. It is surjective since $S_1$ acts transitively on $S_1\slash K$, and given a tangent vector $v\in T_x X$, we can choose a path $\gamma_v$ with $\dot{\gamma}_v(0)=v$. Since $f$ is a fibration, we can lift $\gamma_v$ to a smooth path $\eta_v$ in $M$, and it follows that $f_\ast \dot{\eta}_v(0)=v$. Therefore $f$ is a submersion.\end{proof}
This proves Theorem \ref{thm:expl}.(i).
Now we turn to the proof of Theorem \ref{thm:expl}.(ii). Assume that $G$ has finitely many components and that $S$ is compact. As in the proof of Theorem \ref{thm:expl}.(i), the first step is completely Lie theoretic and aims to find a map $\varphi:\Gamma'\rightarrow \Lambda$ for $\Gamma'\subseteq\Gamma$ a finite index subgroup and $\Lambda$ a lattice in an appropriate nilpotent Lie group $H$. This map naturally extends to a map $G'\rightarrow H $ for a closed connected subgroup $G'\subseteq G^0$. 

As before, the second step is geometric, and uses $\varphi$ to relate the geometry of $M$ to an appropriate locally homogeneous space: The map $\Gamma'\rightarrow\Lambda$ is induced by a homotopy class of maps $M'\rightarrow H\slash\Lambda$. Again, we will select a representative that will be a fiber bundle by an averaging procedure using the structure of nilpotent groups.

\subsection*{Proof of Theorem \ref{thm:expl}.(ii)} Assume that $G$ has finitely many components and that $S$ is compact. Since $G$ has finitely many components, we may assume $\Gamma\subseteq G^0$. 

Consider the Levi decomposition $G^0=RS$. Let
	$$p:G\rightarrow G\slash R=:S'$$ 
be the natural projection. Set $R':=p^{-1}(\overline{p(\Gamma)}^0)$. This will be the Lie group $G'$ alluded to above.
	\begin{lemnr}
		\label{lem:fi}
		$\Gamma\cap R'$ has finite index in $\Gamma$.
	\end{lemnr}
\begin{proof}
$\overline{p(\Gamma)}$ is closed inside the compact group $S'$, so that $\overline{p(\Gamma)}$ is compact. Therefore $\overline{p(\Gamma)}$ has finitely many components. So $p(\Gamma\cap R')=p(\Gamma)\cap \overline{p(\Gamma)}^0$ has finite index in $p(\Gamma)=p(\Gamma)\cap \overline{p(\Gamma)}$. Therefore $\Gamma\cap R'$ has finite index in $\Gamma$.
\end{proof}

Replace $\Gamma$ by the finite index subgroup $\Gamma\cap R'$. Note that $R'$ is an extension
	$$1\rightarrow R\rightarrow R'\rightarrow \overline{p(\Gamma)}^0\rightarrow 1.$$
Theorem \ref{thm:auslsolcl} implies that $\overline{p(\Gamma)}^0$ is solvable (see also \cite[4.1.7(ii)]{ovlie}). Hence $R'$ is solvable-by-solvable, so $R'$ is itself solvable. Let $N'$ be the nilradical of $R'$. We consider two cases according to whether or not $N'$ is cocompact in $R'$. Note that these cases are not disjoint: If $R'$ is abelian, both methods below work. In fact the obtained maps are essentially the same.

\textit{Case 1: Noncocompact case}. Suppose $N'$ is not cocompact in $R'$. Since $R'\slash N'$ is abelian and noncompact, there is a torus $T$ and $n\geq 1$ such that 
	$$R'\slash N'\cong T\times \mathbb{R}^n.$$
Set $H:=(R'\slash N')\slash T\cong \mathbb{R}^n$ and let $\varphi:R'\rightarrow H$ be the natural projection. By Theorem \ref{thm:mostn}, we have that $\Gamma\cap N'$ is a lattice in $N'$. Hence $\Lambda:=\varphi(\Gamma)$ is a lattice in $\mathbb{R}^n$. There is a homotopy class of maps 
	$$M\rightarrow H\slash\Lambda$$
inducing $\varphi|_\Gamma$ on $\pi_1$. Let $f$ be any representative in this homotopy class and let $h:\widetilde{M}\rightarrow H$ be defined by
	$$h(x):=\int_{R'\slash\Gamma} \varphi(g) \widetilde{f}(g^{-1}x)d\mu(g),$$
where $\mu$ is induced by Haar measure on $R'$. Using invariance of Haar measure, it is easy to see that $h$ is $\varphi$-equivariant. This finishes the proof of Theorem \ref{thm:expl}.(ii) in this case.

\textit{Case 2: Cocompact case}. Suppose $N'$ is cocompact in $R'$. As above, we know that $\Gamma\cap N'$ is a lattice in $N'$. By Proposition \ref{prop:latther} the image of $\Gamma$ is a lattice in the compact group $R'\slash N'$. Hence $\Gamma\cap N'$ has finite index in $\Gamma$, and we can replace $\Gamma$ by $\Gamma\cap N'$.

Since $\Gamma\cap [N',N']$ is a lattice in $[N',N']$, we could now again split up into cases according to whether $[N',N']$ is cocompact in $N$, and repeat the procedure of Case 1 in case it is not, and pass to $\Gamma\cap [N',N']$ if it is, etc. Some quotient of successive terms of the lower central series of $N'$ is guaranteed to be noncompact since $N'$ itself is noncompact. This argument shows that $M$ virtually fibers locally equivariantly over a torus. However, below we will show that a stronger conclusion holds, namely that $M$ actually fibers over a nilmanifold very closely related to $N'$.

Let $T$ be the unique maximal compact torus of $N'$. By Theorem \ref{thm:nilptorus}, $T$ is central and the group $H:=N'\slash T$ is simply-connected and nilpotent. We claim that $M$ virtually fibers locally equivariantly over $H$. Let
	$$\varphi:N'\rightarrow H$$
be the natural projection, and let $\Lambda:=\varphi(\Gamma)$ be the image of $\Gamma$ in $H$. It follows from Theorem \ref{thm:expdiff} that $H$ is contractible. Therefore there is a homotopy class of maps
	$$M\rightarrow H\slash\Lambda$$
inducing $\varphi|_\Gamma$ on $\pi_1$. Pick a representative $f_1$ and lift it to a map
	$$\widetilde{f}_1:\widetilde{M}\rightarrow H.$$
In general $\widetilde{f}_1$ need not be $\varphi$-equivariant, but only $\varphi|_\Gamma$-equivariant. To construct a $\varphi$-equivariant map, set for $k\geq 1$:
	$$\Gamma_k:=\langle\gamma\in N' \mid \gamma^{2^k}\in\Gamma\rangle$$
and $\Lambda_k:=\varphi(\Gamma_k)$. We claim that $\Lambda_k$ is discrete. By Theorem \ref{thm:lattsubgp}, there exists an exponentiated lattice $\Lambda'$ of $H$ that contains $\Lambda$ with finite index. Then 
	$$L_k':=\exp^{-1}\{h\in H\mid h^{2^k} \in \Lambda'\}=\{X\in \frh\mid 2^k X\in \exp^{-1}\Lambda'\}$$
is discrete. Further, the condition $2^k X\in\exp^{-1}\Lambda'$ is linear in $X$ since $\Lambda'$ is an exponentiated lattice. It follows that $L_k'$ is a lattice in $\frh$, so that 
	$$\Lambda_k':=\exp L_k'=\langle h\in H\mid h^{2^k}\in \Lambda'\rangle$$ 
is a lattice in $H$. Since $\Lambda_k\subseteq \Lambda_k'$, it follows that $\Lambda_k$ is also a lattice in $H$. By carrying out the same argument for $\Gamma$ on the universal cover $\widetilde{N}'$ of $N'$, it follows that $\Gamma_k$ is discrete.

Further the sequence
	$$\Gamma_1\subseteq\Gamma_2\subseteq\Gamma_3\subseteq\dots$$
is increasing with dense union in $N'$. Now we would like to say that contractibility of $H$ implies that there is a homotopy class of maps
	$$\widetilde{f}_k: M_k:=\widetilde{M}\slash\Gamma_k\rightarrow H\slash\Lambda_k$$
that induces $\varphi|_{\Gamma_k}$ on $\pi_1$. However, $\Gamma_k$ may have torsion, so $M\rightarrow M_k$ may not be a covering map. Consider instead $\widetilde{X}:=\widetilde{M}\slash T$ and let 
	$$\pi:\widetilde{M}\rightarrow\widetilde{X}$$
be the natural projection. Because $T$ is connected and $\widetilde{M}$ is simply-connected, it follows $\widetilde{X}$ is simply-connected (see e.g. \cite[II.6.3]{brtrgps}). Further $H$ acts on $\widetilde{X}$. Because $T$ is compact and $N'$ acts on $\widetilde{M}$ properly, the action of $H$ on $\widetilde{X}$ is also proper. Since $\Lambda_k$ is a lattice in a simply-connected nilpotent Lie group, it is torsion-free, so $\widetilde{X}$ covers $X_k:=\widetilde{X}\slash\Lambda_k$ for every $k$. Then we can choose continuous maps
	$$h_k:X_k\rightarrow  H\slash\Lambda_k.$$
Note that since $X$ is not necessarily a manifold, we can only require $h_k$ to be continuous, not necessarily smooth. We can lift $h_k$ to a $\Lambda_k$-equivariant map
	$$\widetilde{h}_k:\widetilde{X}\rightarrow H.$$
Choose a basepoint $x_0\in \widetilde{X}$ and fundamental domains $B_1\supseteq B_2\supseteq \dots$ for the action of $\Lambda_1\subseteq \Lambda_2\subseteq\dots $ on $H$. Then we can arrange the lifts so that $\widetilde{h}_k(x_0)\in B_k$ for all $k$.
	\begin{lemnr}
	\label{lem:conv}
	$\{\widetilde{h}_k\}_k$ converges uniformly to some $\widetilde{h}:\widetilde{X}\rightarrow H$.
	\end{lemnr}
\begin{proof}
Note that $A_m:=\widetilde{h}_m^{-1}(B_m)$ is a fundamental domain for the action of $\Lambda_m$ on $\widetilde{X}$ containing $x_0$. For every $k\geq 1$, we know that $h_k$ is $\Lambda$-equivariant, so it suffices to check the uniform convergence on $A_1$. We note that $\widetilde{X}=\widetilde{M}\slash T$ is naturally a metric space because $T$ is compact, and $H$ acts isometrically on $\widetilde{X}$. Therefore $A_1$ is a compact metric space, so it suffices to check $(\widetilde{h}_k)_k$ is a Cauchy sequence in the uniform topology. 

Since $\Lambda_k$ form an increasing sequence of lattices in $H$ with dense union, we have $\diam B_k\rightarrow 0$ monotonically as $k\rightarrow\infty$. Now let $\varepsilon>0$ and choose $N\geq 1$ such that diam$B_N<\varepsilon$. Let $n\geq m\geq N$ and $x\in A_1$. Since $\widetilde{h}_n, \widetilde{h}_m$ are both $\Lambda_m$-equivariant, we can without loss of generality assume that $x$ belongs to $A_m$. Further we can choose $\lambda\in \Lambda_n$ such that $\lambda^{-1} x\in A_n$. Hence we have
	\begin{align*}
		d(\widetilde{h}_n(x),\widetilde{h}_m(x))&=d(\widetilde{h}_n(x),\widetilde{h}_n(\lambda x_0))+d(\widetilde{h}_n(\lambda x_0),\widetilde{h}_m(x))\\
		&=d(\widetilde{h}_n(\lambda^{-1}x),\widetilde{h}_n(x_0))+d(\widetilde{h}_n(\lambda x_0),\widetilde{h}_m(x))\\
		&\leq \diam B_n + \diam B_m\\
		&\leq 2\diam B_N.
	\end{align*}

Since $\diam B_n\rightarrow 0$ as $n\rightarrow\infty$, it follows that $\widetilde{h}_k$ converge uniformly.\end{proof}
It is clear that $\widetilde{h}$ is $\Lambda_k$-equivariant for every $k$. Since the union of $\Lambda_k$ is dense in $H$, $\widetilde{h}$ is actually $H$-equivariant. Now set
	$$\widetilde{f}:=\widetilde{h}\circ\pi.$$
Then $\widetilde{f}$ is $\varphi$-equivariant. However, note that we cannot expect $\widetilde{f}$ to be smooth, for the maps $h_k$ were only continuous. Note that $H$ is linear, because it is a simply-connected nilpotent group, so that by Theorem \ref{thm:sea} we can equivariantly homotope $\widetilde{f}$ to a smooth equivariant map. 

Finally, to show that $f$ is a fiber bundle, we note that the proof of Lemma \ref{lem:fiber} still applies in the current situation. The only change that needs to be made is to consider the fibration $N'\rightarrow H$ instead of $G^0\rightarrow S_1\slash K$. Otherwise the proof applies verbatim. This proves Theorem \ref{thm:expl}.(ii).

We will now prove Theorem \ref{thm:riemsub}.
\begin{proof}[Proof of Theorem \ref{thm:riemsub}] Suppose $p:M\rightarrow X\slash\Lambda$ is a locally equivariant fibering with lift $\widetilde{p}:\widetilde{M}\rightarrow X$ that is equivariant with respect to a homomorphism $\varphi:G\rightarrow H$, and $H$ acts transitively on $X$. We want to show that $p$ is a Riemannian submersion with totally geodesic fibers in some Riemannian metric on $M$.

Let $L$ be the stabilizer of some point $x_0\in X$. Then $X$ is $H$-equivariantly diffeomorphic to $H\slash L$. We claim first that the structure group of $\widetilde{p}$ can be reduced to $\varphi^{-1}(L)$. Let $\pi:H\rightarrow X$ be the natural projection. Then $\varphi\circ\pi: G\rightarrow X$ is a fiber bundle with fibers $\varphi^{-1}(L)$. Therefore we have an open covering $\{U_\alpha\}_\alpha$ of $X$ with local trivializations
	$$\chi_\alpha:\varphi^{-1}(L)\times U_\alpha\rightarrow (\pi\circ\varphi)^{-1}(U_\alpha)$$
of $\varphi\circ\pi$. Any such local trivialization $\chi_\alpha$ naturally induces a local trivialization 
	$$\xi_\alpha:\widetilde{p}^{-1}(eL)\times U_\alpha\rightarrow \widetilde{p}^{-1}(U_\alpha)$$
of $\widetilde{p}$ over $U_\alpha$ defined by
	$$\xi_\alpha(x,y):=\chi_\alpha(e,y)x.$$
It is easy to check that $\xi_\alpha$ is actually a local trivialization and is precisely multiplication by $\chi_\alpha(e,y)\in G$ when restricted to the fiber over $y$. Therefore the structure group induced by this collection of trivializations is contained in $G$ and preserves a given fiber, say $\widetilde{p}^{-1}(eL)$. It follows the structure group is contained in $\varphi^{-1}(L)$.

Let $g$ denote the Riemannian metric on $\widetilde{M}$. The tangent spaces to the fibers of $\widetilde{p}$ give a $G$-invariant vertical distribution $\ker \widetilde{p}_\ast$, and restricting $g$ to $\ker\widetilde{p}_\ast$ induces a Riemannian metric $g_V$ on each fiber of $\widetilde{p}$. Further, since $G$ acts isometrically on $(\widetilde{M},g)$, it follows that for any $x\in X$ and $h\in G$,  the restriction of $h$ to the fiber over $x$ is an isometry
	$$h|_{\widetilde{p}^{-1}(x)}: (\widetilde{p}^{-1}(x),g_V)\rightarrow (\widetilde{p}^{-1}(\varphi(h)x),g_V).$$
In particular, the structure group $\varphi^{-1}(L)$ of $\widetilde{p}$ acts isometrically with respect to $g_V$. To extend $g_V$ to a Riemannian metric on $\widetilde{M}$, consider the orthogonal complement $\mathcal{H}:=(\ker\widetilde{p}_\ast)^\perp$ of $\ker\widetilde{p}_\ast$. Then $\mathcal{H}$ is a $G$-invariant distribution on $\widetilde{M}$, and $\widetilde{p}_\ast$ maps $\mathcal{H}$ isomorphically onto the tangent bundle of $X$. Let $h$ denote the Riemannian metric on $X$. Since $\widetilde{p}_\ast$ is an isomorphism on $\mathcal{H}$, we can pull back $h$ to a Riemannian metric $g_H:=\widetilde{p}^\ast h$ on $\mathcal{H}$. Note that since $\varphi(G)$ acts isometrically on $X$, it follows that $G$ preserves $g_H$. 

Therefore $g_V\oplus g_H$ is a $G$-invariant Riemannian metric on $\widetilde{M}$. In this metric, $\widetilde{p}$ is a Riemannian submersion. Vilms proves \cite{vilmsharm} that $\widetilde{p}$ has totally geodesic fibers. Finally, being a Riemannian submersion with totally geodesic fibers is a local condition, so the result follows for $p:M\rightarrow X\slash\Lambda$.
\end{proof}

\section{Decomposition of manifolds with many local symmetries}
\label{sec:decomp}

In this section, we prove Corollary \ref{cor:ind}, which follows from repeated application of (the proof of) Theorem \ref{thm:expl}.

\begin{proof}[Proof of Corollary \ref{cor:ind}] Let $M$ be a closed Riemannian manifold and assume that $G:=\isom(\widetilde{M})$ has finitely many components. Because $G$ has finitely many components, we can replace $\Gamma$ by $\Gamma\cap G^0$. As usual, let $G^0=RS$ be the Levi decomposition. The proof consists of three steps. In the first step, we construct a virtually locally equivariant fibering $p_1$ of $M$ over a locally symmetric space of noncompact type if $S$ is noncompact. A fiber $F_1$ of $p$ will almost satisfy the hypotheses of Theorem \ref{thm:expl}.(ii). In fact, $F_1$ precisely satisfies the hypotheses if $Z(S)$ is finite. In the second step we construct a virtually locally equivariant fibering $p_2$ of $F_1$ over a torus. If $Z(S)$ is infinite, some care needs to be taken to apply the proof of Theorem \ref{thm:expl}.(ii). A fiber $F_2$ of $p_2$ will again almost satisfy the hypotheses of Theorem \ref{thm:expl}.(ii), and we construct a virtually locally equivariant fibering of $F_2$ over a nilmanifold.

\textit{Step 1 (fibering over symmetric space)}:  If $S$ is noncompact, by Theorem \ref{thm:expl}.(i) there exists a locally symmetric space of noncompact type $X$ associated to the Lie group $(S\slash L)\slash Z$ and a fiber bundle $p_1:M'\rightarrow X$ that is equivariant with respect to the natural projection $\varphi:G^0\rightarrow (S\slash L)\slash Z$, where $M'$ is a finite cover of $M$.

If $S$ is compact, fix the following notation: Set $M':=M$, let $X=\ast$ be a point, $p_1:M'\rightarrow X$ the constant map, and $\varphi:G^0\rightarrow 1$ be the trivial morphism.

\textit{Step 2 (fibering over Euclidean space)}: Let $F_1$ be a fiber of $p_1$. Note that regardless of (non)compactness of $S$, we have a natural Riemannian metric on $F_1$ by restricting the Riemannian metric on $M$ to the tangent bundle of $F_1$. Further, by applying the long exact sequence of homotopy groups for the fibration $\widetilde{p}_1: \widetilde{M}\rightarrow \widetilde{X}$ and using that $\widetilde{X}$ is contractible, we see that a fiber of $\widetilde{p}_1$ is simply-connected. Therefore for any $x\in\widetilde{X}$, the fiber $\widetilde{p}_1^{-1}(x)$ over $x$ is isometric to $\widetilde{F}_1$.

Under this identification, $\ker\varphi$ is a closed subgroup of Isom$(\widetilde{F}_1)$ containing $\pi_1(F_1)=\ker\varphi\cap\pi_1(M')$. If $Z$ is infinite then $\ker\varphi$ has infinitely many components. In fact, as we will see below, the components of $\ker\varphi$ are indexed by a finite index subgroup of $Z$. Further, $(\ker\varphi)^0$ has compact Levi subgroups isogenous to $L$, so unfortunately we cannot apply Theorem \ref{thm:expl} directly to $F_1$. However, we will show that the extension
	\begin{equation}
		1\rightarrow RL\rightarrow \ker\varphi\rightarrow Z\rightarrow 1
		\label{eq:ses3}
	\end{equation}
almost splits as a direct product, and the ideas of the proof of Theorem \ref{thm:expl} will then apply:
	\begin{claim} There is a finite index subgroup $G_1$ of $\ker\varphi$ containing $RL$ such that the extension
	\begin{equation}
	1\rightarrow RL\rightarrow G_1\rightarrow Z_1\rightarrow 1
	\end{equation}
splits as a direct product, where $Z_1$ is the image of $G_1$ in $Z$. 
	\end{claim}
	\begin{proof}[Proof of claim] Note that $Z$ is finitely generated and abelian, so we can pass to a finite index torsion-free subgroup $Z'$. First we lift a finite index subgroup $Z_2$ of $Z'\subseteq S\slash L$ to a subgroup of $G$. As before, write $S=LH$, where $H$ is the maximal noncompact factor. Then there is an exact sequence
	\begin{equation}
	1\rightarrow RL\cap H\rightarrow Z(H)\overset{q}{\rightarrow} Z''\rightarrow 1
	\end{equation}
	for a finite index subgroup $Z''$ of $Z$ and $q$ the restriction of the natural projection $G\rightarrow S\slash L$. Note that $RL\cap H$ is a finite abelian group. Now consider the restriction of this sequence to $Z_2:=Z'\cap Z''$ as follows
	$$1\rightarrow RL\cap H\rightarrow q^{-1}(Z_2)\rightarrow Z_2\rightarrow 1.$$
	This extension is trivial because $q^{-1}(Z_2)\subseteq Z(H)$ is abelian and $Z_2$ is torsion-free. Therefore $Z_2$ lifts to a subgroup of $H$, which we also denote by $Z_2$. It remains to show that a finite index subgroup $Z_1$ of $Z_2$ centralizes $RL$. To see this, consider the action of $H$ on $RL$. It induces an action on the Lie algebra $\mathfrak{r}\rtimes\mathfrak{l}$ of $RL$, giving a map
	$$H\rightarrow GL(\mathfrak{r}\rtimes\mathfrak{l}).$$
Since semisimple groups with infinite center are not linear (see e.g. \cite[7.9]{knlie}), we find that a finite index subgroup $Z_1$ of $Z_2$ acts trivially on $\mathfrak{r}\rtimes\mathfrak{l}$. Since $RL$ is connected, any automorphism of $RL$ that is trivial on $\mathfrak{r}\rtimes\mathfrak{l}$ is trivial on $RL$. It follows that there is a finite index subgroup $G_1$ of $\ker\varphi$ such that
	$$G_1\cong RL\times Z_1.$$
	\end{proof}
The proof of Theorem \ref{thm:expl}.(ii) applies verbatim in this situation and yields a closed connected subgroup $R'$ of $G_1^0=RL$. Assume first that the $R'$ has noncompact abelianization. Then the proof of Case 1 in Theorem \ref{thm:expl}.(ii) shows that there is a fiber bundle $p_2:F_1'\rightarrow T^n$ for a finite cover $F_1'$ of $F_1$ and a torus $T^n$. Further, $p_2$ is equivariant with respect to the composition
	$$\psi:R'\times Z_1\rightarrow R'\rightarrow \bbR^n.$$
Here the first map is projection onto the first factor.

If $R'$ has compact abelianization, let $p_2:F_1\rightarrow\ast$ be the map to a point and $\psi:G_1\rightarrow 1$ be the trivial morphism. Then regardless of (non)compactness of the abelianization of $R'$, we see that $F_2$ is naturally a Riemannian manifold and $\ker\psi$ acts isometrically on $\widetilde{F}_2$. Every fiber of $\widetilde{p}_2$ is simply-connected by the long exact sequence on homotopy groups applied to the fibration $p_2$, so that we can identify $\widetilde{F}_2$ with a fiber of $\widetilde{p}_2$. Under this identification, $\ker\psi$ is a closed subgroup of Isom$(\widetilde{F}_2)$ containing $\pi_1(F_2)$. Further $N:=(\ker\psi)^0$ is a nilpotent group. 

\textit{Step 3 (fibering over nilmanifold):} The proof of Case 2 of Theorem \ref{thm:expl}.(ii) constructs a fiber bundle $p_3:F_2'\rightarrow \overline{N}$ for a finite cover $F_2'$ of $F_2$ and a nilmanifold $\overline{N}$, and $p_3$ is equivariant with respect to the composition
	$$\chi:\ker\psi\cong N\times Z_1\rightarrow N\rightarrow N\slash T,$$
where $T$ is the unique maximal closed normal subgroup of $N$, $N\slash T$ is the universal cover of $\overline{N}$, and the first map is projection onto the first factor. This construction proves Assertions (i)-(iii) of Corollary \ref{cor:ind}. It remains to prove Assertion (iv), that Isom$(\widetilde{F}_3$ is compact if $Z$ is finite.

To see this, note that by the long exact sequence on homotopy groups and using that $N\slash T$ is contractible, we have $\pi_1(F_3)=\ker\chi\cap\pi_1(F_2)$. In particular, $\pi_1(F_3)$ contains $Z_1$ as a finite index subgroup. Therefore if $Z$ is finite, the fibers $F_3$ of $p_3$ have finite fundamental group, so that the universal cover $\widetilde{F}_3$ is compact. Hence in this case Isom$(\widetilde{F}_3)$ is compact.
\end{proof}

\section{Manifolds with uniform and nonuniform quotients}
\label{sec:quots}

In this section, we prove Corollary \ref{cor:quots}, namely that Riemannian symmetric spaces are the only source of manifolds with uniform and nonuniform quotients.
\begin{proof} Write $\Gamma:=\pi_1(M)$ and $\Lambda:=\pi_1(N)$. Consider the group $\Delta:=\langle\Gamma,\Lambda\rangle$ acting on $X$. Let $\overline{\Delta}$ be the closure of $\Delta$ in $\isom(X)$. Set $\Gamma_0:=\Gamma\cap \overline{\Delta}^0$ and similarly define $\Lambda_0$ and $\Delta_0$. Since $\Delta$ is dense in $\overline{\Delta}$ and $\overline{\Delta}\slash\overline{\Delta}^0$ is discrete, we know that $\Delta$ maps surjectively to $\overline{\Delta}\slash\overline{\Delta}^0$. Since $\Delta$ is generated by $\Gamma$ and $\Lambda$, it follows that
	$$\overline{\Delta}\slash\overline{\Delta}^0=\langle\Gamma\slash\Gamma_0,\Lambda\slash\Lambda_0\rangle.$$
We claim that both $\Gamma\slash\Gamma_0$ and $\Lambda\slash\Lambda_0$ have finite index in $\overline{\Delta}\slash\overline{\Delta}^0$. To see this, note that $X$ admits a $\Delta$-invariant measure $\mu$ induced by the metric. This measure descends to $N$ and on the other hand we can also push it forward under the natural projection
	$$q:X\rightarrow X\slash\overline{\Delta}^0.$$
$\mu$ induces a measure $\nu$ on $X\slash\langle \overline{\Delta}^0,\Lambda\rangle$ by pushing forward along
	$$q:N=X\slash\Lambda\longrightarrow X\slash\langle \overline{\Delta}^0,\Lambda\rangle=(X\slash\overline{\Delta}^0)\slash(\Lambda\slash\Lambda_0).$$
We can also obtain $\nu$ in a different way. Namely, note that the measure $q_\ast\mu$ (on $X\slash\overline{\Delta}^0$) descends to $X\slash\langle\overline{\Delta}^0,\Lambda\rangle$. Since $\overline{\Delta}\slash\overline{\Delta}^0$ acts properly discontinuously on $X\slash\overline{\Delta}^0$, we can see by the same argument that $q_\ast\mu$ descends to $X\slash\overline{\Delta}$. The map
	$$q_\Delta: X\slash\langle\overline{\Delta}^0,\Lambda\rangle \rightarrow X\slash\overline{\Delta}$$
is between finite volume spaces (with respect to the measures defined above) and the ratio of the volumes is precisely $[\overline{\Delta}\slash\overline{\Delta}^0:\Lambda\slash\Lambda_0]$. Hence $\Lambda\slash\Lambda_0$ is finite index in $\overline{\Delta}\slash\overline{\Delta}^0$. Similarly it follows that $\Gamma\slash\Gamma_0$ has finite index in $\Delta\slash\Delta_0$.

However, since $X\slash\Lambda$ is noncompact but $X\slash\overline{\Delta}$ is compact, it follows that $\Lambda$ is not cocompact in $\overline{\Delta}$. By the above claim, we must have that $\Lambda_0\subseteq \overline{\Delta}^0$ is not cocompact. We claim it is still a lattice. To see this, write
	$$q_{\overline{\Delta}^0}:X\slash\Lambda\rightarrow X\slash\langle\overline{\Delta}^0,\Lambda\rangle$$
for the natural projection with fibers homeomorphic to $\overline{\Delta}^0\slash\Lambda_0$. We can disintegrate the natural volume form along the fibers of $q_{\overline{\Delta}^0}$. Let $\nu_x$ denote the measure on the fiber $q_{\overline{\Delta}^0}^{-1}(x)=K_x\backslash \overline{\Delta}^0\slash\Lambda_0$ where $K_x$ is the stabilizer of a point in the $\overline{\Delta}^0$-orbit. Since the volume form $\mu$ is $\overline{\Delta}^0$-invariant after lifting to the universal cover, it is easy to see that for a.e. $x$, the measure $\nu_x$ is $\overline{\Delta}^0$-invariant on $K_x\backslash\overline{\Delta}^0$. Therefore $\nu_x$ is either finite for a.e. $x$ or infinite for a.e. $x$. Since $X\slash\Lambda$ has finite volume, we must have that $\nu_x$ is finite for a.e. $x$. In particular $\Lambda_0$ has finite covolume in $\overline{\Delta}^0$.

We now show that $\overline{\Delta}^0$ contains a noncompact semisimple subgroup. Write the Levi decomposition $\overline{\Delta}^0=RS$ where $R$ is the solvable radical and $S$ is a Levi subgroup. Then $\overline{\Lambda_0 R}^0$ is solvable by Theorem \ref{thm:auslsolcl}. If $S$ is compact, then the image $\overline{\Lambda_0 R}^0\slash R$ in $\overline{\Delta}^0\slash R$ is compact, so it has finitely many components. Therefore $\Lambda_0$ has a finite index subgroup that is a lattice in $R$. However, any lattice in a solvable group is cocompact, hence $\Lambda_0$ is cocompact in $\overline{\Delta}^0$, but this is impossible by the claim above. So we must have that $S$ is noncompact.

Therefore $\overline{\Delta}^0$ contains a noncompact semisimple subgroup, and hence so does $\isom(X)^0$. Therefore we can apply Theorem \ref{thm:expl}.(i) to conclude that $M$ virtually locally equivariantly fibers over a locally symmetric space of noncompact type, as desired.\end{proof}

\section{Cautionary examples}
\label{sec:ex}
In \cite{FW} Farb-Weinberger prove the following theorem.
	\begin{thmnr}[Farb-Weinberger] Let $M$ be a closed, aspherical Riemannian manifold. Then either Isom$(\widetilde{M})$ is discrete, or $M$ is isometric to an orbibundle
		$$F\rightarrow M\rightarrow B$$
	where
		\begin{itemize}
			\item $B$ is a good Riemannian orbifold, and
			\item each fiber $F$ is isometric (with respect to the induced metric) to a nontrivial closed, aspherical locally homogeneous space.
		\end{itemize}
	\label{thm:fw}
	\end{thmnr}
Here a \emph{Riemannian orbifold} is a second countable Hausdorff topological space locally modeled on spaces of the form $V\slash G$, where $G$ is a finite group and $V$ is a linear $G$-representation equipped with a $G$-invariant Riemannian metric. An orbifold is \emph{good} if $B$ is the global quotient of a manifold by a properly discontinuous group action. An \emph{orbibundle} is a map with fibers $F$ that locally can be expressed as a projection
	$$F\times_G V\rightarrow V\slash G.$$
In fact, the proof of \cite{FW} shows that the fibers $F$ are projections to $M$ of Isom$(\widetilde{M})^0$-orbits in $\widetilde{M}$. On the other hand, in the results above, we express a finite cover of $M$ as a fiber bundle over a locally homogeneous base. The following example shows that this difference is genuine and that in the nonaspherical situation a result of the form of Theorem \ref{thm:fw} is not possible.
	\begin{exnr}
	Let $H$ be the three-dimensional real Heisenberg group, and let $\Gamma\subseteq H$ be the integer Heisenberg group. Then $Z(\Gamma)\cong \mathbb{Z}$. Set $N:=H\slash Z(\Gamma)$ and let $\Lambda\cong \mathbb{Z}^2$ be the image of $\Gamma$ in $N$. Let $Z:=Z(N)\cong S^1$.
	
	Now let $Z$ act on $S^2$ by rotations around some fixed axis, and form the space
		$$X:=S^2\times_Z N.$$
	Then $X$ is diffeomorphic to $S^2\times\bbR^2$ and $N$ naturally acts on $X$ via
		$$g[(p,h)]:=[(p,gh)].$$
	Note that if $p\in S^2$ is not fixed by $Z$, then any point of the form $[p,h]\in X$ is not fixed by any element of $N$. On the other hand, if $p\in S^2$ is fixed by $Z$, then $[p,h]$ is also fixed by $Z$.

	It is easy to construct Riemannian metrics on $X$ such that $N$ acts isometrically.  Then $M:=X\slash\Lambda$ is a closed manifold such that Isom$(\widetilde{M})$ is not discrete (in fact has finitely many components). However, it is easy to check that any nondiscrete closed subgroup $G$ of Isom$(\widetilde{M})$ containing $\Lambda$ in fact contains $Z$. 
	
	To see this, project $G$ to $N\slash Z=\bbR^2$. If the image of $G^0$ is trivial, then $G^0\subseteq Z$. Since $G^0\neq 1$, it follows that $G^0=Z$. On the other hand, if the image is nontrivial, it contains a line, say $\mathbb{R}v$ for some $0\neq v\in\bbR^2$. Write $v=(v_1,v_2)\in\bbR^2$. Further choose $(n,m)\in\bbZ^2$ such that $mv_1 - nv_2\neq 0$. Then we can choose $\widetilde{v}(t)\in N$ that projects to $tv$ and $\lambda\in\Lambda$ that projects to $(n,m)\in \bbZ^2$. An easy computation shows that
		$$[\widetilde{v}(t),\lambda]=\exp(2\pi i t(m v_1-n v_2))\in Z,$$
	and hence $G$ contains $Z$.
	
	Therefore there are two types of $G$-orbits in $M$: The principal orbits are diffeomorphic to $G$ and the singular orbits are diffeomorphic to $G\slash Z$. Therefore there is no orbibundle
		$$F\rightarrow M\rightarrow B$$
	where the fibers $F$ are projections of $G$-orbits in $\widetilde{M}$. On the other hand, the natural projection $X\rightarrow N\slash Z$ descends to a locally equivariant fibering $M\rightarrow T^2$ with fibers diffeomorphic to $S^2$.
	\label{ex:fwfails}
	\end{exnr}
Example \ref{ex:fwfails} shows that for nonaspherical manifolds, it is necessary to consider virtually locally equivariant fiberings instead of Riemannian orbibundles. A useful property of the Riemannian orbibundles produced by Theorem \ref{thm:fw} is that the fibers are isometric. This still holds true for the fibers of $\widetilde{p}:\widetilde{M}\rightarrow X$ if $p$ is a virtually locally equivariant fibering of $M$ by equivariance of $\widetilde{p}$ with respect to $\varphi$. Here it is crucial that $\varphi(H)$ acts transitively on $X$. However, this property does not necessarily descend to a compact quotient of $\widetilde{M}$. In the next example we construct a closed Riemannian manifold $M$ that virtually fibers locally equivariantly, but there is no Riemannian metric such that the fibering $p$ is locally equivariant with isometric fibers.
	\begin{exnr} Let $S$ be a semisimple group of noncompact type with a faithful linear representation on $\bbR^n$ that is not by isometries in any Riemannian metric, and such that a cocompact torsion-free lattice $\Gamma$ preserves a lattice $\Lambda$ in $\bbR^n$. Explicit constructions are classical (see \cite{wmarithm}). For example, let $\mathcal{O}$ be an order in a central division algebra $D$ over $\bbQ$ such that $D\otimes_\bbQ \bbR\cong M_{n\times n}(\bbR)$. Then $\Gamma=\mathcal{O}\cap \SL(n,\bbR)$ is a cocompact lattice in $\SL(n,\bbR)$. Further $\SL(n,\bbR)$ acts on the vector space $M_{n\times n}(\bbR)$ such that $\Gamma$ preserves the lattice $\mathcal{O}$. This action is not by isometries in any Riemannian metric because the action is not proper.
	
	Form the group $G:=\bbR^n\rtimes S$ induced by this representation and consider the action of $G$ on $X:=G\slash K$ by left-translations, where $K$ is the maximal compact subgroup of $S$. Equip $G$ with a left-invariant Riemannian metric that is also right-$K$-invariant. This metric naturally induces a metric on $X$ such that $G$ acts isometrically on $X$. 
	
	Let $\varphi:G\rightarrow S$ be the natural projection. Then $X$ fibers $\varphi$-equivariantly over $Y:=S\slash K$, and the fibers are isometric to Euclidean spaces $\bbR^n$.

Note that $\Lambda\rtimes\Gamma$ is a torsion-free cocompact lattice in $G$. Then $M:=(\Lambda\rtimes \Gamma)\backslash X$ is an aspherical Riemannian manifold and $G\subseteq \isom(\widetilde{M})$. The natural projection $X\rightarrow Y$ induces a locally equivariant fibering $p:M\rightarrow N$, where $N$ is the locally symmetric space $N:=\Gamma\backslash S\slash K$.
	
We claim that for any Riemannian metric on $M$ such that $p:M\rightarrow N$ is a locally equivariant fibering, the fibers of $p$ are not all isometric. To see this, suppose $H\subseteq G$ is a closed subgroup containing a finite index subgroup $\Delta$ of $\Lambda\rtimes \Gamma$, and such that $\varphi(H)$ acts transitively on $Y$.
	
	Let $\langle\cdot,\cdot\rangle$ be the metric on the fiber $\widetilde{p}^{-1}(eK)$ over $eK\in Y$. Then for $g\in S$, the metric on the fiber $\widetilde{p}^{-1}(gK)$ is given by $h_\ast \langle \cdot, \cdot \rangle$ where $h\in H$ is such that $\varphi(h)=g$. 
	
	For the fibers of $p$ we have $p^{-1}(\Gamma g K)\cong \Lambda\backslash \widetilde{p}^{-1}(gK)$. This is a torus with the metric induced by $h_\ast \langle \cdot,\cdot\rangle$ where $h$ is as above, and is isometric to the torus $(h\Lambda h^{-1})\backslash \widetilde{p}^{-1}(eK)$ with the metric $\langle \cdot, \cdot\rangle$. Because $\varphi(H)=S$ and $S$ does not act by isometries on $\bbR^n\cong \widetilde{p}^{-1}(eK)$, we can choose $h\in H$ such that the lattices $\Lambda$ and $h\Lambda h^{-1}$ in $\bbR^n$ give nonisometric tori. This proves the claim.
	
	\label{ex:isomfiber}
	\end{exnr}
Since Problem \ref{prob:symm} is naturally a problem about Riemannian manifolds, one might hope that the fiber bundle obtained in Theorem \ref{thm:expl} is a Riemannian submersion. This is true after possibly changing the metric on $M$ as shown in Theorem \ref{thm:riemsub}. The next example shows that in general the change of metric is necessary.
\begin{exnr} Set $M:=T^n \times S^1$ for some $n\geq 2$. Let 
	$$f:S^1\rightarrow\bbR_{>0}$$ 
be a smooth nonconstant function on $S^1$. Define a Riemannian metric $g$ on $M$ by the rule
	$$g_{(x,y)} := f(y) g_1 \oplus g_2$$
for $(x,y)\in M=G\slash\Gamma\times N$, $g_1$ a bi-invariant metric on $T^n$ and $g_2$ a Riemannian metric on $S^1$. It is not hard to choose $g_2$ such that $\isom(\widetilde{M})^0=G$. Then $\bbR^n$ acts isometrically on $\widetilde{M}$ by acting on the first factor. Choose a smooth map
	$$b:S^1\rightarrow T^n$$
and lift $b$ to a map $\widetilde{b}:\bbR\rightarrow \bbR^n$. Then we obtain a locally equivariant fiber bundle $p:M\rightarrow T^n$ by setting $p(g,y)=b(y)g$. It is not hard to see that every equivariant fiber bundle $\widetilde{M}\rightarrow G$ is of this form. We aim to show that none of these maps is a Riemannian submersion.

First note that if $b$ is constant then the fibers of $p$ are vertical, so that the orthogonal complement to the fibers is tangent to the foliation by $G$-orbits. In this case, $p$ is not a Riemannian submersion because $f$ is not constant.

If $b$ is not constant, an explicit computation shows that $p$ is a Riemannian submersion at the point $(e,x)$ if and only if for every $(v_1, v_2)$ in the orthogonal complement of $T_{(e,x)} p^{-1}(\widetilde{b}(x))$, we have
\begin{equation}
f(x)\left(\|v_1\|^2+\left\langle l_{\widetilde{b}(x)\ast} v_1, \widetilde{b}_\ast v_2\right\rangle\right)=\|v_1\|^2+\|\widetilde{b}_\ast v_2\|^2+2\left\langle l_{\widetilde{b}(x)\ast}v_1,\widetilde{b}_\ast v_2\right\rangle,
\label{eq:riemsub}
\end{equation}
where $l_{\widetilde{b}(x)}$ is left-translation by $\widetilde{b}(x)$ and all inner products are with respect to $g_1$. Note that the orthogonal complement to the tangent space to a fiber has dimension $n$. Therefore it intersects the tangent space to the $\bbR^n$-orbit in a subspace of dimension at least $n-1\geq 1$. For $(v_1,0)$ belonging to this subspace, Equation \ref{eq:riemsub} simplifies to
	$$f(x)||v_1||^2=||v_1||^2.$$
Hence $f(x)=1$. If $w\in T_x S^1$ and $w\neq 0$, then a computation yields that $(v,w)\in T_{(e,x)}p^{-1}(\widetilde{b}(x))^\perp$ precisely for 
	$$v=\lambda l_{b(x)^{-1}\ast} b_\ast v_2$$
where $\lambda:=\frac{||w||^2}{||b_\ast w||^2}$. Inserting $(v,w)$ in Equation \ref{eq:riemsub} and using $f(x)=1$, we obtain
	$$(\lambda^2+1) ||b_\ast w||^2=(\lambda^2+2\lambda+1)||b_\ast w||^2$$
which is a contradiction since $\lambda\neq 0$. Therefore $p$ cannot be a Riemannian submersion.
\label{ex:norsub}
\end{exnr}

\section{An equivariant smooth approximation theorem}
\label{sec:sea}

Recall that if $M, N$ are closed manifolds, then every continuous map $M\rightarrow N$ can be approximated by a smooth map. This is essentially an application of the Whitney embedding theorem. Similarly, one can obtain equivariant approximation theorems from equivariant embedding theorems. In order to state such theorems, recall that if $G$ acts on a manifold $M$ and $x\in M$ is a point with isotropy group $G_x$, then the orbit $Gx$ of $x$ is homeomorphic to $G\slash G_x$. The geometry of an orbit is thus determined by the isotropy group of a point, and if $x,y$ lie in the same orbit, then the isotropy groups are conjugate. An \emph{orbit type} is by definition a conjugacy class of closed subgroups of $G$. An equivariant embedding theorem for compact group actions is due to Mostow \cite{mostowemb} and Palais \cite{palaisemb}. 
	\begin{thmnr}[Mostow-Palais Equivariant Embedding] Let $M$ be manifold and $G$ a compact Lie group acting smoothly on $M$ with finitely many orbit types. Then $M$ can be equivariantly and smoothly embedded into a $G$-representation.\end{thmnr}
This can be applied to give the following.
	\begin{cornr} Let $M,N$ be manifiolds and $G$ a compact Lie group acting smoothly on $M,N$, and assume that the action of $G$ on $N$ has only finitely many orbit types. Then any $G$-equivariant continuous map $f:M\rightarrow N$ can be $G$-equivariantly approximated by a $G$-equivariant smooth map.\end{cornr}
Because the ideas of the proof will be used below, we include a sketch.
	\begin{proof}[Sketch of proof, see {\cite[VI.4.2]{brtrgps}}] Let $i:N\hookrightarrow V$ be a $G$-equivariant embedding in a $G$-representation, and approximate $i\circ f:M\rightarrow V$ by a smooth map $g$. Then average $g$ over $G$ to obtain an equivariant smooth map $\overline{g}:M\rightarrow V$. We can choose $\overline{g}$ so that it has image in a tubular neighborhood $U$ of $i(N)$ in $V$. Let $r$ be the projection of $U$ on $i(N)$. Then $i^{-1}\circ r\circ \overline{g}:M\rightarrow N$ is a smooth $G$-equivariant map that approximates $f$.\end{proof}
We want such an equivariant approximation theorem for proper $G$-actions where $G$ is not necessarily compact. Unfortunately, in general such actions cannot be embedded in representations. For example, if $G$ is a nonlinear Lie group, then its action by left translations on itself cannot be embedded in a representation. However, there is an equivariant embedding theorem for proper $G$-actions for linear groups $G$.
	\begin{thmnr}[{Palais \cite{palaisslice}}] Let $M$ be a manifold and $G$ a linear group acting on $M$ smoothly, properly and with finitely many orbit types. Then $M$ can be equivariantly and smoothly embedded in a $G$-representation.\end{thmnr}
A second problem in generalizing the equivariant approximation theorem to proper actions is that we cannot average a map over $G$ if $G$ is not compact. However, if $G$ contains a cocompact lattice $\Gamma$ such that the action of $\Gamma$ on $M$ is a covering action with compact quotient, then an averaging procedure works, as we now explain.
	\begin{thmnr} Let $M,N$ be manifolds and $G,H$ be connected Lie groups such that $G$ acts smoothly on $M$ and $H$ acts smoothly on $N$. Let $\varphi:G\rightarrow H$ be a smooth homomorphism. Assume that $H$ is linear and the $H$-action on $N$ has only finitely many orbit types. Further assume that $G$ has a discrete subgroup $\Gamma$ such that $\Gamma$ acts freely and cocompactly on $M$ and $\varphi(\Gamma)$ acts freely on $N$.
	
	Then every $\varphi$-equivariant continuous map $f:M\rightarrow N$ can be $\varphi$-equivariantly approximated by a $\varphi$-equivariant smooth map.
	\label{thm:sea}
	\end{thmnr}
\begin{proof} Palais' equivariant embedding theorem for linear groups implies that there exists a $H$-equivariant smooth embedding $i:N\rightarrow V$ of $N$ in an $H$-representation $V$.

The map $f:M\rightarrow N$ descends to a map
	$$\overline{f}:M\slash\Gamma\rightarrow N\slash\Lambda$$
where $\Lambda:=\varphi(\Gamma)$. We can approximate $\overline{f}$ by a smooth map
	$$\overline{h}:M\slash\Gamma\rightarrow N\slash\Lambda$$
and lift to $h:M\rightarrow N$ that is $\varphi|_\Gamma$-equivariant. Now average $$i\circ h:M\rightarrow V$$ with respect to $G\slash\Gamma$, i.e.\ set
	$$k(x):=\int_{G\slash\Gamma} \varphi(g)^{-1}(i\circ h)(gx)dg.$$
We can choose $\overline{h}$ such that $i\circ h$ has image in a tubular neighborhood $U$ of $i(N)\subseteq V$. Then the projection back to $i(N)$ gives a $\varphi$-equivariant smooth map $M\rightarrow N$.\end{proof}

\bibliographystyle{alpha}
\bibliography{localsymmbib}

\begin{thebibliography}{Kna02}

\bibitem[Bre72]{brtrgps}
G.~Bredon.
\newblock {\em Introduction to Compact Transformation Groups}.
\newblock Academic Press, 1972.

\bibitem[Ebe80]{eblatt}
P.~Eberlein.
\newblock Lattices in spaces of nonpositive curvature.
\newblock {\em Ann. of Math.}, 111:435--476, 1980.

\bibitem[Ebe82]{ebisom}
P.~Eberlein.
\newblock Isometry groups of simply connected manifolds of nonpositive
  curvature ii.
\newblock {\em Acta Math.}, 149:41--69, 1982.

\bibitem[Fra94]{frharm}
S.~Frankel.
\newblock Locally symmetric and rigid factors for complex manifolds via
  harmonic maps.
\newblock {\em Ann. of Math.}, 139:285--300, 1994.

\bibitem[FW08]{FW}
B.~Farb and S.~Weinberger.
\newblock {I}sometries, rigidity and universal covers.
\newblock {\em Ann. of Math.}, 168:915--940, 2008.

\bibitem[Kna02]{knlie}
A.~Knapp.
\newblock {\em Lie Groups Beyond an Introduction}.
\newblock Birkh{\"a}user, 2002.

\bibitem[Mel09]{melsslorentz}
K.~Melnick.
\newblock Compact lorentz manifolds with local symmetry.
\newblock {\em J. Differential Geom.}, 81(2):355--390, 2009.

\bibitem[Mor]{wmarithm}
D.~Witte Morris.
\newblock {\em Introduction to Arithmetic Groups}.

\bibitem[Mos57]{mostowemb}
G.~Mostow.
\newblock Equivariant embeddings in euclidean space.
\newblock {\em Ann. of Math.}, 65:432--446, 1957.

\bibitem[OV00]{ovlie}
A.~Onishchik and E.~Vinberg.
\newblock {\em Lie Groups and Lie Algebras, vol. 2: Discrete Subgroups of Lie
  Groups and Cohomologies of Lie Groups and Lie Algebras}.
\newblock Springer-Verlag, 2000.

\bibitem[Pal57]{palaisemb}
R.~Palais.
\newblock Imbedding of compact differentiable transformation groups in
  orthogonal representations.
\newblock {\em J. Math. Mech.}, 6:673--678, 1957.

\bibitem[Pal61]{palaisslice}
R.~Palais.
\newblock On the existence of slices for actions of non-compact lie groups.
\newblock {\em Ann. of Math.}, 73:295--323, 1961.

\bibitem[Rag72]{raghlie}
M.~Raghunathan.
\newblock {\em Discrete Subgroups of Lie Groups}.
\newblock Springer-Verlag, 1972.

\bibitem[SY79]{synpc}
R.~Schoen and S.~Yau.
\newblock Compact group actions and the topology of manifolds with nonpositive
  curvature.
\newblock {\em Topology}, 18:361--380, 1979.

\bibitem[Vil70]{vilmsharm}
J.~Vilms.
\newblock Totally geodesic maps.
\newblock {\em J. Differential Geom.}, 4:73--79, 1970.

\bibitem[Wu88]{wunote}
T.~Wu.
\newblock A note on a theorem on lattices in lie groups.
\newblock {\em Canad. Math. Bull.}, 31(2):190--193, 1988.

\end{thebibliography}

\end{document}